\documentclass[12pt]{article}
\usepackage{amsmath}
\usepackage{amssymb}
\usepackage{amsmath,amssymb,amsbsy,amsfonts,amsthm,latexsym,
amsopn,amstext,amsxtra,euscript,amscd}

\begin{document}

\def\A{\mathbb{A}}
\def\B{\mathbf{B}}
\def \C{\mathbb{C}}
\def \F{\mathbb{F}}
\def \K{\mathbb{K}}

\def \Z{\mathbb{Z}}
\def \P{\mathbb{P}}
\def \R{\mathbb{R}}
\def \Q{\mathbb{Q}}
\def \N{\mathbb{N}}
\def \Z{\mathbb{Z}}

\def\B{\mathcal B}
\def\e{\varepsilon}

\def\cA{{\mathcal A}}
\def\cB{{\mathcal B}}
\def\cC{{\mathcal C}}
\def\cD{{\mathcal D}}
\def\cE{{\mathcal E}}
\def\cF{{\mathcal F}}
\def\cG{{\mathcal G}}
\def\cH{{\mathcal H}}
\def\cI{{\mathcal I}}
\def\cJ{{\mathcal J}}
\def\cK{{\mathcal K}}
\def\cL{{\mathcal L}}
\def\cM{{\mathcal M}}
\def\cN{{\mathcal N}}
\def\cO{{\mathcal O}}
\def\cP{{\mathcal P}}
\def\cQ{{\mathcal Q}}
\def\cR{{\mathcal R}}
\def\cS{{\mathcal S}}
\def\cT{{\mathcal T}}
\def\cU{{\mathcal U}}
\def\cV{{\mathcal V}}
\def\cW{{\mathcal W}}
\def\cX{{\mathcal X}}
\def\cY{{\mathcal Y}}
\def\cZ{{\mathcal Z}}

\def\f{\frac{|\A||B|}{|G|}}
\def\AB{|\A\cap B|}
\def \Fq{\F_q}
\def \Fqn{\F_{q^n}}

\def\({\left(}
\def\){\right)}
\def\fl#1{\left\lfloor#1\right\rfloor}
\def\rf#1{\left\lceil#1\right\rceil}
\def\Res{{\mathrm{Res}}}

\newcommand{\comm}[1]{\marginpar{
\vskip-\baselineskip \raggedright\footnotesize
\itshape\hrule\smallskip#1\par\smallskip\hrule}}

\newtheorem{lem}{Lemma}
\newtheorem{lemma}[lem]{Lemma}
\newtheorem{prop}{Proposition}
\newtheorem{proposition}[prop]{Proposition }
\newtheorem{thm}{Theorem}
\newtheorem{theorem}[thm]{Theorem}
\newtheorem{cor}{Corollary}
\newtheorem{corollary}[cor]{Corollary}
\newtheorem{prob}{Problem}
\newtheorem{problem}[prob]{Problem}
\newtheorem{ques}{Question}
\newtheorem{question}[ques]{Question}
\newtheorem{rem}{Remark}
\newtheorem{defin}{Definition}

\title{Multiplicative decomposition of arithmetic progressions in prime fields}

\author{{M. Z. Garaev}
\\
\normalsize{Centro de Ciencias Matem\'{a}ticas,}
\\
\normalsize{Universidad Nacional Aut\'onoma de M\'{e}xico,}\\
\normalsize{Morelia 58089, Michoac\'{a}n, M\'{e}xico}\\
\normalsize{\tt garaev@matmor.unam.mx}\\
\and\\
{S.~V.~Konyagin}\\
\normalsize{Steklov Mathematical Institute,}\\
\normalsize {Gubkin Street 8,}\\
\normalsize {Moscow, 119991, Russia}\\
\normalsize{\tt konyagin@mi.ras.ru} }

%\small Mathematics Subject Classifications: 11L05}

\date{\empty}

\pagenumbering{arabic}

\maketitle

\begin{abstract}
We prove that there exists an absolute constant $c>0$ such that if
an arithmetic progression $\cP$ modulo a prime number $p$ does not
contain zero and has the cardinality less than $cp$, then it can not
be represented as a product of two subsets of cardinality greater
than $1$, unless $\cP=-\cP$ or $\cP=\{-2r,r,4r\}$ for some residue
$r$ modulo $p$.
\end{abstract}
\maketitle

%\newpage

\section{Introduction}

Let $\F_p$ be the field of residue classes modulo a prime number
$p.$ Given two sets $\cA, \cB\subset \F_p$ their sum-set $\cA+\cB$
and product-set $\cA\cB$ are defined as
$$
\cA+\cB=\{a+b ; \, a\in \cA,\, b\in \cB\}; \quad \cA\cB=\{ab; \, a\in \cA,\, b\in \cB\}.
$$
A set $\cS\subset\F_p$ is said to have a nontrivial additive
decomposition, if
$$
\cS=\cA+\cB
$$
for some sets $\cA\subset\F_p,\, \cB\subset\F_p$ with $|\cA|\ge
2,\,|\cB|\ge 2.$

The problem of nontrivial additive decomposition of multiplicative
subgroups of $\F_p^{*}$ has been recently investigated in several
works. S\'ark\"ozy~\cite{Sark} conjectured that the set of quadratic
residues has no nontrivial additive decomposition and obtained
several results on this problem. Further progress in this direction
has been made by Shkredov~\cite{Shkr} and Shparlinski~\cite{Shpr}.

In the present paper we are interested in multiplicative
decomposition of intervals in $\F_p$. This problem has been
investigated by Shparlinski~\cite{Shpr}. Let
$$\cI=\{n+1,
n+2,\ldots,n+N\}\pmod p
$$
be an interval in $\F_p$. Shparlinski observed that if $N<(p-1)/32$
and if there is a decomposition $\cI=\cA\cB$, $|\cA|\ge2,
|\cB|\ge2,$ then Bourgain's sum-product estimate (see,
Lemma~\ref{lem:Bourg} below) leads to the sharp bound
$$
N\le |\cA||\cB|\le 32 N.
$$
Here, for a positive integer $k$ and a set $\cX$, the notation
$k\cX$ is used to denote the $k$-fold sum of $\cX$, that is
$$
k\cX=\{x_1+\ldots+x_k;\quad x_i\in \cX\}.
$$

Clearly, if $0\in \cS\subset\F_p$, then one has the decomposition
$\cS=\{0,1\}\cS$. On the other hand, if for such a set $\cS$ we have
$\cS\setminus\{0\}=\cA\cB$, then it follows that
$\cS=\{\cA\cup\{0\}\}\cB$. We give the following definition.

\begin{defin}
We say that the set $\cS\subset\F_p$ has a nontrivial multiplicative
decomposition if
\begin{equation}
\label{eqn:S=AB def}
\cS\setminus\{0\}=\cA\cB
\end{equation}
for some sets $\cA\subset\F_p,\,\cB\subset\F_p$ with $|\cA|\ge 2,\,
|\cB|\ge 2.$
\end{defin}

Any nonzero set $\cS$ with $\cS=-\cS$ admits a nontrivial
decomposition, namely~\eqref{eqn:S=AB def} holds with $\cA=\{-1,1\},
\cB=\cS\setminus\{0\}$. We also note that for $p\ge 5$ we have the
following decomposition of a special interval of $3$ elements:
$$
\{3^*-1,3^*,3^*+1\}\pmod p=\{-1,2\} \cdot \{-3^*,1-3^*\}\pmod p.
$$
Here and below  $n^*$ denotes the multiplicative inverse of $n$
modulo $p$.

In the present paper we prove the following statement.

\begin{theorem}
\label{thm:main} There exists an absolute constant $c>0$ such that
if an interval $\cI\subset\F_p$ of cardinality $|\cI|<cp$ has a
nontrivial multiplicative decomposition, then either
$$
\cI=\pm \{3^*-1,3^*,3^*+1\}\pmod p,
$$
or
$$
\cI=-\cI.
$$
 In the latter case
any nontrivial decomposition $\cI\setminus\{0\} =\cA\cB$ implies
that one of the sets $\cA$ or $\cB$ coincides with $\{-r,r\}$ for
some residue class $r\in\F_p$.
\end{theorem}

The following statement shows that the constant $c$ in the condition
of Theorem~\ref{thm:main} can not be taken $c=1/2$.
\begin{theorem}
\label{thm:p=1 mod 4} Let $p\equiv 1\pmod 4$, $p\not=5$. Then for
any integer $L$ satisfying
$$
\frac{p-1}{2}\le L\le p-1
$$
the interval
$$
\cI=\{1,2,\ldots, L\}\pmod p
$$
admits a nontrivial multiplicative decomposition.
\end{theorem}

For arbitrary prime $p\ge 3$, we have the following result.

\begin{theorem}
\label{thm:k1k2} Let $p\ge 3$ and $k_1,k_2$ be  integers satisfying
$$
k_1\ge 0.4(p-1),\quad k_2\ge 0.4(p-1).
$$
Then the interval
$$
\cI=\{n\in\Z; \, -k_1\le n\le k_2\}\pmod p
$$
admits a nontrivial multiplicative decomposition.
\end{theorem}

\bigskip

The idea behind the proof of Theorem~\ref{thm:main} is as follows.
As Shparlinski, we use Bourgain's sum product estimate. Here, we
apply it to the sets $k\cA$ and $\cB$ for  a suitable integer $k$
(which can be as large as constant times $p/(|\cA||\cB|)$). Then we
use some arguments from additive combinatorics and show that the set
$\cA$ (and $\cB$) forms a positive proportion of some arithmetic
progression modulo $p$. Using this information we eventually reduce
our problem to its analogy in $\Q$ (the set of rational numbers).

Throughout the paper some absolute constants are indicated
explicitly in order to make the arguments more transparent.

\section{The case of rational numbers}

\begin{lem}
\label{lem:Rationals} Let $\cP\subset\Q$ be a finite arithmetic
progression such that
$$
\cP\setminus\{0\}=\cA\cB
$$
for some sets $\cA\subset \Q,\, \cB\subset \Q$ with $|\cA|\ge
2,\,|\cB|\ge 2$. Then there exist rational numbers $r, r_1,r_2$ such
that either
$$
\cA=\{-r_1,2r_1\},\quad \cB=\{-r_2,2r_2\},
$$
or one of the sets $\cA$ or $\cB$ coincides with the set $\{-r,r\}$.
\end{lem}

\begin{proof}

Assume contrary, let $\cP\setminus\{0\}=\cA\cB$ be such that the
sets $\cA$ and $\cB$ do not satisfy the conclusion of the lemma. We
dilate the set
 $\cA$  such that the new set $\cA'$ consists on integers that are relatively  prime.
 Similarly we construct the set $\cB'$. Then  $\cA'\cB'$ is also a set
 of integers that are relatively prime and we have $\cP'\setminus\{0\}=\cA'\cB'$,
 where $\cP'\subset\Q$ is an arithmetic progression ($\cP'$ is a dilation of $\cP$).
 Let $d$ be the difference of this progression.
 We rewrite $\cA=\cA', \cB=\cB', \cP=\cP'$.

\bigskip

Let $a_1, a_2\in \cA$. For any element $b\in \cB$ we have
$(a_1-a_2)b\in \cP-\cP$, implying $(a_1-a_2)b\equiv 0\pmod d.$ It
then follows from the construction of $\cB$ that $a_1-a_2\equiv
0\pmod d$. Thus, the set $\cA$ is contained in a progression with
difference $d$. Analogously the set $\cB$ is  contained in a
progression with difference $d$.

\bigskip

Define $a_0$ and $b_0$ to be the maximal by absolute value elements
of $\cA$ and $\cB$ correspondingly. Without loss of generality we
can assume that $a_0>0$, $b_0>0$. Since $\cA\not=\{-a_0,a_0\},\,
\cB\not=\{-b_0,b_0\}$, we have
\begin{equation}
\label{eqn:a_0>1, b_0>1}
a_0\ge \max\left\{2, \frac{d+1}{2}\right\}, \quad b_0\ge \max\left\{2,\frac{d+1}{2}\right\}.
\end{equation}
Then $a_0b_0$ is the largest element of $\cP$ and  $a_0b_0>d$. From
$$
a_0b_0-d\in \cP\setminus\{0\}=\cA\cB,
$$
it follows that for some $a_1\in\cA,\, b_1\in\cB$ we have
\begin{equation}
\label{eqn:a0b0-d=a1b1}
a_0b_0-d=a_1b_1.
\end{equation}
If $|a_1|=a_0$, then
$$
d=a_0(b_0-|b_1|)\ge (b_0-|b_1|)\frac{d+1}{2}.
$$
This implies that $b_0-|b_1|=1$ and $a_0=d$. Since by the
assumption, $\cA\not=\{-d,d\}$ and $0\not\in \cA$, we get a
contradiction with the maximality property of $a_0$.

Thus, in~\eqref{eqn:a0b0-d=a1b1} we have $|a_1|\not=a_0$. Similarly,
$|b_1|\not=b_0.$ It then follows that
$$
d=a_0b_0-|a_1||b_1|\ge a_0b_0-(a_0-1)(b_0-1)=a_0+b_0-1.
$$
Combining this with~\eqref{eqn:a_0>1, b_0>1} we get that
$$
a_0=\frac{d+1}{2};\quad b_0=\frac{d+1}{2};\quad d\ge 3.
$$
Therefore, the maximality properties of $a_0$ and $b_0$ imply that
$$
\cA=\cB=\Bigl\{\frac{1-d}{2},\, \frac{1+d}{2}\Bigr\}.
$$
Hence,
$$
\cP\setminus\{0\}=\cA\cB= \Bigl\{\frac{1-d^2}{4},\frac{(1-d)^2}{4},\, \frac{(1+d)^2}{4}\Bigr\}.
$$
Since $(1-d)^2/4\not=d$, we get
$$
\frac{1-d^2}{4}=\frac{(1-d)^2}{4}-d.
$$
This implies $d=3,\,\cA=\cB=\{-1,2\}$ and  concludes the proof of
our lemma.
\end{proof}

\section{Some facts from additive combinatorics}

We need several facts from additive combinatorics.

\begin{lem}
\label{lem:Bourg} Let $\cX\subset\F_p,\,\cY\subset\F_p$ and
$\cX\not=\{0\}, \cY\not=\{0\}$. Then
$$
|8\cX\cY-8\cX\cY|\ge\frac{1}{2}\min\{|\cX||\cY|, p-1\}.
$$
\end{lem}

\begin{lem}
\label{lem:Frei} For a sufficiently large $p$, let $\cX$ be a subset
of $\F_p$ such that $|\cX| < p/35$ and
$$
|2\cX|<\frac{12}{5}|\cX|-3.
$$
Then $\cX$ is contained in an arithmetic progression of at most
$|2\cX|-|\cX|+1$ terms.
\end{lem}

Lemma~\ref{lem:Bourg} is Bourgain's sum product estimate
from~\cite{Bourg}.  Lemma~\ref{lem:Frei} is Freiman's result on
additive structure of sets with small doubling (see, for
example,~\cite[Theorem 2.11]{Nat}).

\begin{lem}
\label{lem:Positive PropAP} Let $\cX$ be a subset of $\F_p$ and $m$
be a positive integer such that
$$
|m\cX|<\min\{33 m|\cX|, \frac{p-1}{8}\}.
$$
Assume that $m\cX$ is contained in an arithmetic progression of at
most $2|m\cX|$ terms. Then $\cX$ is contained in an arithmetic
progression of at most
$132|\cX|$ terms.
\end{lem}

\begin{proof}
By a suitable dilation of the set $\cX$, we can assume that $m\cX$
forms at least a half of an arithmetic progression with difference
equal to $1$. In particular, the diameter of this progression is not
greater than $2|m\cX|$. Thus,
$$
m\cX-m\cX\subset[-2|m\cX|, 2|m\cX|]\pmod p.
$$

Let $x_1,x_2\in \cX$ and $x_1-x_2=d\pmod p$ with $|d|\le (p-1)/2$.
It suffices to prove that $|d|<66|\cX|.$ Observe that all the
elements $id\pmod p$, $1\le i\le m$, are contained in the set
$m\cX-m\cX$. Thus,
$$
\{d, 2d,\ldots, md\}\pmod p \subset[-2|m\cX|, 2|m\cX|]\pmod p.
$$
It then follows that we actually have $id\in [-2|m\cX|, 2|m\cX|]$
for all $1\le i\le m$. Indeed, this is trivial for $i=1$. Assume
that for some $1\le i\le m-1$ we have $id\in [-2|m\cX|, 2|m\cX|]$
and let
\begin{equation}
\label{eqn:ind i d}
(i+1)d\equiv z_{i+1}\pmod p
\end{equation}
for some $z_{i+1}\in [-2|m\cX|, 2|m\cX|]$. Since, by the induction
hypothesis,
$$
|(i+1)d|\le 2|id|\le 4|m\cX|\le (p-1)/2,
$$
the congruence~\eqref{eqn:ind i d} is converted to an equality, as
desired.

In particular, $md\subset[-2|m\cX|, 2|m\cX|],$ implying $|d|\le
\frac{2}{m}|m\cX|<66|\cX|.$
\end{proof}

The following statement is known as the Cauchy-Davenport theorem
(see, for example,~\cite[Theorem 2.2]{Nat}).

\begin{lem}
\label{lem:CauchyDav} For any nonempty subsets $\cX$ and $\cY$ of
$\F_p$ the following bound holds:
$$
|\cX+\cY|\ge \min\{p, |\cX|+|\cY|-1\}.
$$
\end{lem}

We will also need the following simple statement.

\begin{lem}
\label{lem:Dir} Let $0<\delta<1$ and let $L$ be an integer with
$L>\delta^{-1}$. Assume that the set $\cX\subset \{r+1,r+2,\ldots,
r+L\}$ is such that $|\cX|\ge\delta L$. Then for any positive
integer $k$ with $\delta^{-k}<L$ there exist elements $x_1,x_2\in
\cX$ such that
$$
\frac{\delta^{-(k-1)}}{2}\le x_1-x_2<2\delta^{-k}.
$$
\end{lem}

\begin{proof}
We split the interval $[r+1,r+L]$ into $\lceil\delta^{k} L\rceil-1$
subintervals of length
$$
\frac{L-1}{\lceil\delta^{k} L\rceil-1}<\frac{L}{(\delta^{k} L/2)} =2\delta^{-k}.
$$
From the pigeon-hole principle, one of this intervals (denote it by
$\cR$) contains at least
$$
\frac{|\cX|}{\lceil\delta^{k} L-1\rceil}>\frac{\delta L}{\delta^{k} L}=\delta^{-(k-1)}
$$
elements of $\cX$. Therefore, if $x_1$ and $x_2$ are the largest and
the smallest elements of $\cR\cap \cX$ then
$$
2\delta^{-k}>x_1-x_2\ge\max\{1, \delta^{-(k-1)}-1\}\ge \frac{\delta^{-(k-1)}}{2}.
$$

\end{proof}

\section{Proof of Theorem~\ref{thm:main}}

Let $N<cp$ and assume that the interval
$$
\cI=\{n+1,n+2,\ldots,n+N\}\pmod p
$$
is such that $\cI_0=\cI\setminus\{0\}=\cA\cB$ for some subsets
$\cA\subset \F_p,\, \cB\subset \F_p$ with $\min\{|\cA|,|\cB|\}\ge
2$. Here $c$ is a small positive constant (the smallness of the
constant $c$ is at our disposal).

Let $k$ be the largest integer such that $2^k N=2^k|\cI|<p/33$.
Observe that for any positive integer $m$
$$
(m\cA)\cB\subset m\cA\cB\subset m\cI.
$$
Hence, since $\cI$ is an interval, from Lemma~\ref{lem:Bourg} we
get, for any nonnegative integer $\nu\le k$,
\begin{eqnarray*}
\min\{p-1,|2^{\nu}\cA||\cB|\}&\le& 2|8(2^{\nu}\cA)\cB-8(2^{\nu}\cA)\cB|  \\ \le
2|2^{\nu+3}\cI-2^{\nu+3}\cI| &\le& 32\times 2^{\nu} (|\cI|-1)+2  < 33\times 2^{\nu} |\cA||\cB|.
\end{eqnarray*}
In particular, if
$$p-1\le2|8(2^{\nu}\cA)\cB-8(2^{\nu}\cA)\cB|,$$
then we get
$$p-1\le32\times 2^{\nu} (|\cI|-1)+2<32\times p/33.$$
This contradiction shows that actually
$$|2^{\nu}\cA||\cB|\le2|8(2^{\nu}\cA)\cB-8(2^{\nu}\cA)\cB|.$$
Therefore,
$$
|2^{\nu}\cA||\cB|<33\times 2^{\nu} |\cA||\cB|;\quad |\cA||\cB|\le 2|8\cI-8\cI|< 32N\le 32cp.
$$
Thus,
\begin{equation}
\label{eqn:kA<}
|2^{\nu}\cA|< 33\times 2^{\nu} |\cA|  {\rm{\quad for \quad any \quad}} \nu=0,1,2\ldots, k
\end{equation}
and we also have
\begin{equation}
\label{eqn:32N}
N-1\le |\cA||\cB|< 32cp.
\end{equation}

Since $c$ is small, $k$ is large.  From~\eqref{eqn:kA<} we get
$$
\prod_{\ell=4}^{k-1}\frac{|2^{\ell+1}\cA|}{|2^{\ell}\cA|}=\frac{|2^k \cA|}{|2^4\cA|}<33\times2^k.
$$
Hence, since $k$ is large enough, there exists $4\le \ell<k$ such
that
\begin{equation}
\label{eqn:2 ell A<2.1}
|2^{\ell+1}\cA|<2.1\times |2^{\ell}\cA|<\frac{12}{5}|2^{\ell} \cA|-3.
\end{equation}
Here we also used the inequality $|2^{\ell}\cA|\ge |2^{4}\cA|> 10$
which follows from Lemma~\ref{lem:CauchyDav}.

Since $|2^{\ell+1}\cA|<p/32$, Lemma~\ref{lem:CauchyDav} implies that
$|2^{\ell}\cA|<p/35$. Then applying Lemma~\ref{lem:Frei} with
$\cX=2^{\ell}\cA$ we get that the set $2^{\ell} \cA$ forms at least
a half of an arithmetic progression. Therefore,
inequality~\eqref{eqn:kA<} with $\nu=\ell$ and
Lemma~\ref{lem:Positive PropAP} implies that the set $\cA$ is
contained in an arithmetic progression of at most $132|\cA|$ terms.
By completing the progression, we can assume that $\cA$ is contained
in an arithmetic progression of $132|\cA|$ terms.

Analogously, the set $\cB$ is contained in an arithmetic progression
of $132|\cB|$ terms.

\bigskip

We recall that $\cA\cB=\cI\setminus\{0\}$ and, by~\eqref{eqn:32N},
$|\cA||\cB|<32cp$, where $c$ is a small positive constant. We can
dilate $\cA$ and $\cB$ and assume, without loss of generality, that
for some integer $r$
$$
\cA\subset \{r+1,r+2,\ldots, r+132|\cA|\}\pmod p.
$$

We shall now prove that for any element $b\in B$ there are integers
$u$ and $v$ such that
$$
|u|\le 264^2|\cB|; \quad 1\le v\le 264; \quad  b\equiv \frac{u}{v}\pmod p.
$$

Let $K$ be the integer defined from
$$
|\cA|\le 132^K<132|\cA|.
$$
We associate the elements of $\cA$ with their representatives from
the interval $\{r+1,r+2,\ldots, r+132|\cA|\}$. Note that for any
$a_1\in \cA, a_2\in \cA$ we have
$$
(a_1-a_2)b\in \cA\cB-\cA\cB\subset \cI-\cI \subset [-N+1, N-1] \pmod p.
$$
It then follows from Lemma~\ref{lem:Dir} with $\delta=1/132$ that
$$ b\equiv
\frac{u_1}{v_1}\equiv \frac{u_2}{v_2}\equiv \ldots \equiv
\frac{u_K}{v_K}\pmod p,
$$
for some integers  $u_1,\ldots, u_K, v_1,\ldots,v_K$ with
$$
|u_j|<N;\quad \frac{132^{j-1}}{2}\le v_j< 2\times 132^j.
$$
Moreover, we can assume that $\gcd(u_1,v_1)=1.$

We claim that for any $j\in \{1,2,\ldots, K\}$ there exists an
integer $t_j$ such that
\begin{equation}
\label{eqn:uj=tju1}
u_j=t_ju_1;\quad v_j=t_jv_1.
\end{equation}
We prove this by induction on $j.$ The claim is trivial for $j=1.$
Assume that~\eqref{eqn:uj=tju1} is true for some $1\le j\le K-1$.
Then from
$$
\frac{132^{j-1}}{2}\le v_j=t_jv_1\le 264t_j
$$
we have $t_j\ge 132^{j-1}/528$ and therefore
$$
|u_1|=\frac{|u_j|}{t_j}< \frac{N}{t_j}\le \frac{528 N}{132^{j-1}}.
$$
Next, we have
$$
u_1v_{j+1}\equiv v_1 u_{j+1}\pmod p.
$$
The absolute value of the left  hand side is bounded by
$$
\frac{528N}{132^{j-1}}\times 2\times 132^{j+1}\le 1056\times 132^2 cp\le p/3.
$$
The absolute value of the right hand side is bounded by $264N< p/3$.
Thus, our congruence is converted to the equality
$$
u_1v_{j+1}=v_1 u_{j+1}.
$$
Since $\gcd(u_1,v_1)=1$, there is an integer $t_{j+1}$ such that
$$
u_{j+1}=t_{j+1}u_1;\quad v_{j+1}=t_{j+1}v_1.
$$

Thus,~\eqref{eqn:uj=tju1} holds for all $j=1,2\ldots,K$. In
particular, for $j=K$ we have
$$
u_K=t_K u_1;\quad v_K=t_Kv_1.
$$
Therefore,
$$
t_K=\frac{v_K}{v_1}\ge \frac{132^{K-1}}{528}\ge \frac{|\cA|}{264^2}
$$
implying
$$
|u_1|=\frac{|u_K|}{t_K}\le \frac{N-1}{t_K}\le \frac{264^2(N-1)}{|\cA|}\le 264^2|\cB|.
$$
Since we also have $1\le v_1\le 264$, our claim on the structure of
$b\in \cB$ follows from $b\equiv u_1/v_1\pmod p$.

Denote by $\cA'$ and $\cB'$ the dilations of $\cA$ and $\cB$ defined
from
$$
\cA'=\{(264!)^{*}a;\, a\in \cA\};\quad \cB'=\{264!b;\, b\in \cB\}.
$$
We have
$$
\cA'\cB'=\cA\cB=\cI\setminus\{0\}.
$$
Furthermore,
$$
\cB'\subset \{n \in \Z;\, -266!|\cB'|\le n\le 266!|\cB'|-1\}\pmod p.
$$

We shall prove that for any $a\in \cA'$ there are integers $u',v'$
such that
$$
|u'|\le (267!)^2 |\cA'|;\quad 1\le v'\le 267!;\quad a\equiv \frac{u'}{v'}\pmod p.
$$

Let $K'$ be the integer defined from
$$
|\cB'|\le (2\times 266!)^{K'}<(2\times 266!)|\cB'|.
$$
Let $a\in \cA'$. We note that for any $b_1\in \cB', b_2\in \cB'$ we
have
$$
(b_1-b_2)a\in \cA'\cB'-\cA'\cB'\subset \cI-\cI \subset [-N+1, N-1] \pmod p.
$$
As before, from Lemma~\ref{lem:Dir} with $\delta=1/(2\times 266!)$
it follows that
$$
a\equiv
\frac{u'_1}{v'_1}\equiv \frac{u'_2}{v'_2}\equiv \ldots \equiv
\frac{u'_{K'}}{v'_{K'}}\pmod p,
$$
for some integers  $u'_1,\ldots, u'_K, v'_1,\ldots,v'_K$ with
$$
|u'_j|< N;\quad \frac{(2\times 266!)^{j-1}}{2}\le v'_j< 2\times (2\times 266!)^j;\quad \gcd(u'_1,v'_1)=1.
$$
Exactly as before, it follows by induction on $j$, that for any
$j\in\{1,2,\ldots, K'\}$ there is an integer $t'_j$ such that
$$
u'_j=t'_ju'_1;\quad v'_j=t'_jv'_1.
$$
In particular, taking $j=K'$ we get that
$$
t'_{K'}=\frac{v'_{K'}}{v_1'}\ge \frac{(2\times 266!)^{K'-1}}{4\times 266!}\ge \frac{|\cB'|}{(267!)^2}
$$
Thus,
$$
|u_1'|=\frac{|u'_{K'}|}{t_{K'}}\le\frac{ (267!)^2 (N-1)}{|\cB'|}\le(267!)^2 |\cA'|.
$$
Our claim on the structure of $a\in\cA'$ follows from $a\equiv
u_1'/v_1'\pmod p.$

Let now $\cA''$ is the dilation of $\cA'$ defined as
$$
\cA''=\{(267!)! a;\, a\in \cA'\}.
$$
Since $\cA'\cB'=\cI\setminus\{0\}$, it follows that
$\cA''\cB'=\cP\setminus \{0\}$ for some arithmetic progression
$\cP\subset\F_p.$ Note that now we have
$$
\cA''\subset \{n\in\Z; \, |n|\le (268!)!|\cA|\} \pmod p.
$$
Let
$$
\cA'''\subset \{n\in\Z; \, |n|\le (268!)!|\cA|\},\quad
\cB''\subset \{n \in \Z;\, |n|\le 266!|\cB'|\}
$$
be such that
\begin{equation}
\label{eqn:A'''B''}
\cA''=\cA'''\pmod p,\quad \cB'=\cB''\pmod p.
\end{equation}
Then either $\cA'''\cB''$ or $\cA'''\cB''\cup\{0\}$ is a set
$\{x_1,x_2,\ldots,x_{N'}\},$ with integers $x_i$ satisfying
$|x_i|\le (269!)!cp<0.1 p$ and
$$
x_{i+2}-x_{i+1}\equiv x_{i+1}-x_{i}\pmod p; \quad i=1,2,\ldots, N'-2.
$$
Then the congruence is converted to the equality
$$
x_{i+2}-x_{i+1}= x_{i+1}-x_{i}; \quad i=1,2,\ldots, N'-2.
$$
Thus, we have that either $\cA'''\cB''$ or $\cA'''\cB''\cup\{0\}$ is
an arithmetic progression of integers. Since $|A'''|\ge 2,\,
|B''|\ge 2,$ we can apply Lemma~\ref{lem:Rationals}. It follows that
there exists rational numbers $r, r_1,r_2$ such that either
$$
\cA'''=\{-r_1,2r_1\},\quad \cB''=\{-r_2,2r_2\},
$$
or one of the sets $\cA'''$ or $\cB''$ coincides with the set
$\{-r,r\}$. In the latter case~\eqref{eqn:A'''B''} implies that
either $\cA''$ or $\cB'$ coincides with the set $\{-r,r\}\pmod p$
and the result follows from the fact that $\cA$ and $\cB$ are the
dilations of $\cA''$ and $\cB'$ correspondingly.

In former case, for some $h\in\F_p$ we have
$$
\cI\setminus\{0\}=\{-2h,h,4h\}.
$$
It follows that $\{0\}\not\in \cI$ and we get, for some $h_1,$
$$
\cI=\{h_1(3^*-1),h_13^*,h_1(3^*+1)\}\pmod p.
$$
From this it follows that either $h_1(3^*-1)$ and $h_13^*$ or
$h_13^*$ and $h_1(3^*+1)$ are consecutive elements of $\cI$. Thus,
$h_1\in \{-1, 1\}\pmod p$. This finishes the proof of
Theorem~\ref{thm:main}.

\section{Proof of Theorems~\ref{thm:p=1 mod 4} and~\ref{thm:k1k2}}

The proof of Theorems~\ref{thm:p=1 mod 4} and~\ref{thm:k1k2} uses
ideas from~\cite{GKS}.

We first prove Theorem~\ref{thm:p=1 mod 4}. We can assume that
$L<p-1$. Define positive integers $u$ and $v$ from the
representation $p=u^2+v^2$. Let $h$ be an integer defined from
$$
h\equiv u/v\pmod p.
$$
Note that the set
$$
\cA=\{x\in \cI;\quad hx\in \cI\}
$$
is nonempty (indeed $v\in \cA$). Let $\cB=\{1, h\}\pmod p$. Let us
prove that $\cI= \cA\cB.$ Assume contrary. Since $\cA\cB\subset
\cI,$ there is an element
$$
x\in \cI \setminus \cA\cB.
$$
Note that
$$
hx\not\in \cI,\quad h^{*}x\not\in \cI.
$$
Indeed, if $hx\in \cI$, then $x\in\cA$ and thus $x\in \cA\cB$,
contradiction. If $h^{*}x\in \cI$, then from $h(h^*x)=x\in\cI$ it
follows that $h^{*}x\in \cA$ and thus $x\in \cA\cB$, contradiction.

Therefore, for some $1\le s_1\le p-L-1$ and $1\le s_2\le p-L-1$ we
have
$$
hx\equiv -s_1\pmod p;\quad h^{*}x\equiv -s_2\pmod p.
$$
Since $h^2+1\equiv 0\pmod p$, it follows that $s_1+s_2\equiv 0\pmod
p.$ Impossible.

Thus, we have that $\cI= \cA\cB$. In particular,
$$
|\cA|\ge \frac{L}{|\cB|}\ge \frac{p-1}{4}\ge 3,
$$
which shows that the decomposition is nontrivial and finishes the
proof of Theorem~\ref{thm:p=1 mod 4}.

Let us prove Theorem~\ref{thm:k1k2}. Since $\F^*_p=\F^*_p\cdot
\F^*_p$, we can assume that $k_1+k_2<p-1$. In particular, it follows
that $p\ge 11.$

We make the following observation: for any integer $x$ one of the
elements $2x\pmod p$ or  $2^*x\pmod p$ belongs to the interval
$\cI$. Indeed, if $2^*x\not\in \cI,$ it follows that $x\equiv
2n\pmod p$ for some integer $n$ with $k_2<n<p-k_1$. Then
$$
2x\equiv 4n\equiv 4n-2p\pmod p.
$$
Since
$$
-0.4(p-1)<4n-2p<0.4(p-1),
$$
it follows that $2x\pmod p\in\cI.$

Now we repeat the proof of Theorem~\ref{thm:p=1 mod 4}. Let
$$
\cA=\{x\in \cI\setminus\{0\};\quad 2x\in \cI\setminus\{0\}\}.
$$
Since $1\in \cA$, the set $\cA$ is nonempty. Let $\cB=\{1, 2\}\pmod
p$. Let us prove that $\cI\setminus\{0\}= \cA\cB.$ Assume contrary.
Since $\cA\cB\subset \cI\setminus\{0\},$ there is an  element
$$
x\in \{\cI\setminus\{0\}\} \setminus \cA\cB.
$$
If $2x\in \cI\setminus\{0\}$, then $x\in\cA$ and thus $x\in \cA\cB$,
contradiction. If $2^*x\in \cI \setminus\{0\}$, then $2^*x\in \cA$
and thus $x\in \cA\cB$, contradiction.

Then $2x\in \cI\setminus\{0\}$ and $2^*x\in \cI \setminus\{0\}$
which contradicts to the above made observation.

Thus, we have
$$
\cI\setminus\{0\}= \cA\cB,
$$
with $|\cB|=|\{1,2\}|=2$. In particular,
$$
|\cA|\ge \frac{k_1+k_2}{2}\ge 0.4(p-1)>2,
$$
which shows that the decomposition is nontrivial and finishes the
proof of Theorem~\ref{thm:k1k2}.

\end{document}